\documentclass{daj}

\usepackage{amsmath}
\usepackage{amsthm}
\usepackage{amsfonts}
\usepackage{enumerate}
\usepackage{enumitem}

\DeclareMathOperator{\im}{Im}
\DeclareMathOperator{\Id}{Id}

\def\rest#1#2{\mathchoice
              {\setbox1\hbox{${\displaystyle #1}_{\scriptstyle #2}$}
              \restrictionaux{#1}{#2}}
              {\setbox1\hbox{${\textstyle #1}_{\scriptstyle #2}$}
              \restrictionaux{#1}{#2}}
              {\setbox1\hbox{${\scriptstyle #1}_{\scriptscriptstyle #2}$}
              \restrictionaux{#1}{#2}}
              {\setbox1\hbox{${\scriptscriptstyle #1}_{\scriptscriptstyle #2}$}
              \restrictionaux{#1}{#2}}}
\def\restrictionaux#1#2{{#1\,\smash{\vrule height .8\ht1 depth .85\dp1}}_{\,#2}}

\newcommand{\NN}{\mathbb{N}}
\newcommand{\ZZ}{\mathbb{Z}}
\newcommand{\QQ}{\mathbb{Q}}

\newcommand{\RR}{\mathbb{R}}

\newtheorem{Thm}{Theorem}[section]
\newtheorem{Prop}[Thm]{Proposition}
\newtheorem{Lem}[Thm]{Lemma}
\newtheorem{Cor}[Thm]{Corollary}
\newtheorem*{Question}{Question}
\newtheorem*{Lem*}{Lemma}
\newtheorem*{Thm*}{Theorem}
\newtheorem*{Cor*}{Corollary}

\theoremstyle{definition}
\newtheorem*{Def}{Definition}

\theoremstyle{remark}
\newtheorem{Rem}[Thm]{Remark}
\newtheorem{Ex}[Thm]{Example}
\dajAUTHORdetails{%
  title = {Approximate Lattices and Meyer Sets in Nilpotent Lie Groups}, 
  author = {Simon Machado},
  plaintextauthor = {Paul Erdos, Johan Hastad, Laszlo Lovasz, and Andrew Chi-Chih Yao},
  runningtitle = {Approximate Lattices in Nilpotent Lie Groups}, 
}   

\dajEDITORdetails{%
   year={2020},
   number={1},
   received={7 November 2018},   
   revised={18 October 2019},    
   published={18 February 2020},  
   doi={10.19086/da.11886},       
}   

\begin{document}

\begin{frontmatter}[classification=text]


\author[sm]{Simon Machado}

\begin{abstract}
 We show that uniform approximate lattices in nilpotent Lie groups are subsets of model sets. This extends Yves Meyer's theorem about quasi-crystals in Euclidean spaces. To do so we study relatively dense subsets of simply connected nilpotent Lie groups and their logarithms. We then deduce a simple criterion for the existence of an approximate lattice in a given nilpotent Lie group.
\end{abstract}
\end{frontmatter}

\section{Introduction}
 
Approximate subgroups were defined by Terence Tao (\cite[Definition $3.8$]{MR2501249}) to generalise results from additive combinatorics such as Freiman's theorem and the sum-product phenomenon. This led to the structure theorem for finite approximate subgroups \cite[Theorem $2.10$]{MR3090256} by Emmanuel Breuillard, Ben Green and Terence Tao. However, aiming at a similar structure theorem for infinite approximate subgroups seems hopeless. The class of infinite approximate subgroups being too large, one needs to consider infinite approximate subgroups that satisfy additional assumptions.  Some results for particular classes of infinite approximate subgroups can be found in \cite{MR3438951}, \cite{MR3345797} and \cite{MR2833482}.

In \cite{bjorklund2016approximate}, Michael Bj\"orklund and Tobias Hartnick defined a new class of infinite approximate subgroups called \emph{uniform approximate lattices}: approximate subgroups of locally compact groups that are both uniformly discrete and relatively dense. They aimed at giving a non-commutative generalisation to the mathematical quasi-crystals that were defined in early works of Yves Meyer (\cite{meyer1972algebraic}). Mathematical quasi-crystals are versatile objects that find applications in many areas of mathematics and physics: algebraic numbers (\cite{meyer1972algebraic}), aperiodic order (\cite{MR3136260}) or crystallography (\cite{MR1420414}). Another interesting aspect of mathematical quasi-crystals is that they can be defined in a number of ways. In particular, one can think of quasi-crystals as approximate subgroups with additional topological properties, or in a more geometric approach by \emph{model sets} and \emph{cut-and-project schemes} that link quasi-crystals to lattices. The equivalence of these definitions (\cite[Theorem $IV$]{meyer1972algebraic}) is at the very foundation of the theory of mathematical quasi-crystals (see for instance the survey \cite{moody1997meyer}).

In the same article (\cite{bjorklund2016approximate}) Michael Bj\"orklund and Tobias Hartnick developed a non-commutative theory of model sets and cut-and-project schemes. More precisely, a model set is defined as the projection to the first factor of a certain well chosen subset of a lattice in a product $G \times H$. We refer the reader to Section \ref{definitions} below for precise definitions. Rephrased in the terminology of \cite{bjorklund2016approximate} Meyer's theorem becomes:

\begin{Thm*}[Theorem $IV$, \cite{meyer1972algebraic}]
 Let $\Lambda \subset G$ be an approximate lattice in a locally compact abelian group. Then $\Lambda$ is contained in a model set. 
\end{Thm*}

Michael Bj\"orklund and Tobias Hartnick naturally raised the following question:

\begin{Question}[Problem $1$, \cite{bjorklund2016approximate}]
 Are all uniform approximate lattices of locally compact second countable groups subsets of model sets?
\end{Question}

The main purpose of this article is to answer affirmatively in the particular case of nilpotent Lie groups. 
\begin{Thm}\label{Théorème principal}
 Let $\Lambda \subset G$ be a uniform approximate lattice in a simply connected nilpotent Lie group. Then there exists a unique simply connected nilpotent Lie group $H$ such that:
 \begin{enumerate}[label=(\roman*)]
  \item $\langle \Lambda \rangle$ is isomorphic as an abstract group to a lattice $\Gamma$ in $G \times H$;
  \item There is a compact neighbourhood $W_0$ of $e$ in $H$ such that $$\Lambda \subset \pi_G(\left(G \times W_0\right) \cap \Gamma).$$ 
 \end{enumerate}
 Here $\langle \Lambda \rangle$ denotes the subgroup generated by $\Lambda$ and $\pi_G$ the projection to the first factor. 
\end{Thm}

As a corollary to Theorem \ref{Théorème principal} we get a `Meyer type' theorem for all connected nilpotent Lie groups (i.e. without the assumption of simple connectedness). 

\begin{Cor}\label{Théorème de Meyer pour les groupes de Lie nilpotents}
 Let $\Lambda \subset G$ be a uniform approximate lattice in a connected nilpotent Lie group. Then there exists a simply connected nilpotent Lie group $H$ such that $\Lambda$ is a Meyer set given by a cut-and-project scheme $(G,H,\Gamma)$. 
\end{Cor}

In order to prove Theorem \ref{Théorème principal} we first prove a rigidity result for group homomorphisms defined over uniform approximate lattices. In the case of lattices of nilpotent Lie groups Theorem \ref{Lemme de rigidité} reduces to the Malcev rigidity theorem (see Theorem \ref{Rigidité de Malcev}). 

\begin{Thm}\label{Lemme de rigidité}
Let $\Lambda \subset G$ be a uniform approximate lattice in a simply connected nilpotent Lie group and $\Gamma: = \langle \Lambda \rangle$ be the subgroup generated by $\Lambda$. Moreover, let $f: \Gamma \rightarrow N$ be an abstract group homomorphism from $\Gamma$ to a simply connected nilpotent Lie group $N$. Then there are unique group homomorphisms $\tilde{f}: G \rightarrow N$ and $\rho: \Gamma \rightarrow N$ such that :
\begin{enumerate}[label=(\roman*)]
 \item $\tilde{f}$ is a Lie group homomorphism;
 \item $f=\rest{\tilde{f}}{\Gamma} \cdot \rho$;
 \item The image of $\rho$ lies in the centraliser $C(\im(\tilde{f}))$ of $\im(\tilde{f})$;
 \item $\rest{\rho}{\Lambda}$ is bounded. 
\end{enumerate}

\end{Thm}

\medbreak

Finally, using Theorem \ref{Théorème principal} we prove a criterion for the existence of uniform approximate lattices in simply connected nilpotent Lie groups. Our result is reminiscent of Malcev's criterion for the existence of uniform lattices in simply connected nilpotent Lie groups. Recall Malcev's criterion:

\begin{Thm}[\cite{mal1949class}]\label{CNS existence d'un réseau uniforme}
 Let $G$ be a simply connected nilpotent Lie group and $\mathfrak{g}$ be its Lie algebra.  There is a uniform lattice in $G$ if and only if $\mathfrak{g}$ has a basis with rational structure constants. 
\end{Thm}

In the case of approximate lattices the field of interest is no longer $\QQ$ but the field of real algebraic numbers $\RR \cap \overline{\QQ}$. 

\begin{Thm}\label{Existence d'un réseau approximatif}
 Let $G$ be a simply connected nilpotent Lie group and $\mathfrak{g}$ its Lie algebra. The group $G$ contains an approximate lattice if and only if $\mathfrak{g}$ has a basis with structure constants in $\RR \cap \overline{\QQ}$.
\end{Thm}

Let us give a quick overview of the proof strategy for Theorems \ref{Théorème principal}, \ref{Lemme de rigidité} and \ref{Existence d'un réseau approximatif}. Starting with Theorem \ref{Lemme de rigidité}, we first use the Baker-Campbell-Hausdorff formula to reduce the proof to showing a similar result on Lie algebras. Theorem \ref{Lemme de rigidité} is then a consequence of a well-known result from the abelian theory (see Lemma \ref{Lemme de rigidité abélien}).

Moving on to the proof of Theorem \ref{Théorème principal}, we follow a proof strategy due to Jean-Pierre Schreiber in his proof of \cite[Theorem $2$]{schreiber1973approximations} (see also \cite{moody1997meyer} for a more compact exposition). In order to follow this strategy we rely on the theory of simply connected nilpotent Lie groups developed by Anatoly Malcev in \cite{mal1949class}. We start by studying the Malcev completion of the subgroup generated by the approximate subgroup $\Lambda$. We then show, using Theorem \ref{Lemme de rigidité}, that the Malcev completion is the direct product $G \times H$ we are looking for.

Finally, to prove Theorem \ref{Existence d'un réseau approximatif} we rely on the fact that, according to Theorems \ref{Théorème principal} and \ref{CNS existence d'un réseau uniforme}, the Lie algebra of $G$ is a direct factor of a Lie algebra with rational structure constants. Then we use structure results about Lie algebras and algebraic groups to conclude. The tools are hence fairly different from those used in the previous proofs. 
 
 \section{Definitions}\label{definitions}
  \subsection{Uniform Approximate Lattices}
 
 Although approximate lattices can be defined in greater generality (see \cite[Section $4$]{bjorklund2016approximate}), we will focus on \emph{uniform} approximate lattices. 
 
 First of all, let us fix notations. For subsets $X$ and $Y$ of a group $G$ and $n$ a positive integer, define 
 \begin{align*}
  X\cdot Y:=\left\{xy|  x \in X, y \in Y \right\},  \qquad\qquad\qquad\qquad\\
  X^n:=\left\{x_1\cdots x_n | x_1,\ldots,x_n \in X \right\}  \text{ and } X^{-1}:=\left\{x^{-1} | x \in X \right\}.
 \end{align*}
 We can now define approximate subgroups. 
 
 \begin{Def}[Definition $3.8$, \cite{MR2501249}]
 A subset $\Lambda$ of group $G$ is called an \emph{approximate subgroup} if $\Lambda^{-1}=\Lambda$, the identity element belongs to $\Lambda$ and there is a finite subset $F \subset G$ such that $\Lambda \cdot \Lambda \subset F \cdot \Lambda $. 
 \end{Def}
 
 \begin{Def}
  A subset $X$ of a group $G$ is said to be \emph{symmetric} if $X=X^{-1}$.
 \end{Def}

 A uniform approximate lattice is an approximate subgroup satisfying further topological conditions.
 
 \begin{Def}
  A subset $X$ of a locally compact group is:
  \begin{itemize}
   \item a \emph{relatively dense} set if there is a compact set $K \subset G$ such that $\Lambda \cdot K = G$;
   \item a \emph{uniformly discrete} set if there is a compact neighbourhood $K$ of the identity element in $G$ such that $\forall g \in G, |g\cdot K \cap \Lambda|\leq 1$;
   \item a \emph{Delone} set if it is both relatively dense and uniformly discrete. 
  \end{itemize}

 \end{Def}
 
 \begin{Rem}\label{Caractérisation des ensembles uniforméments discrets}
  Equivalently, $X$ is uniformly discrete if and only if $e$ is not an accumulation point in $X^{-1}\cdot X$.  
 \end{Rem}

 We are now able to define uniform approximate lattices.
 
 \begin{Def}
  A subset $\Lambda$ of a locally compact group $G$ is a \emph{uniform approximate lattice} if:
  \begin{enumerate}[label=(\roman*)]
   \item $\Lambda$ is an approximate subgroup;
   \item $\Lambda$ is a Delone set. 
  \end{enumerate}

 \end{Def}

\begin{Ex}
\begin{enumerate}
 \item  All uniform lattices are uniform approximate lattices.
 
 \item According to a result of de Bruijn, the set of vertices of the Penrose rhombus tiling $P3$ is a uniform approximate lattice in the plane, see \cite{deBruijn1981algebraicI} and \cite{deBruijn1981algebraicII}. 
 
 \item If $\gamma \in \overline{\QQ}$ is a Pisot number, set $X$ as $\{\sum\limits_{i \in I}\gamma^i| I \subset \NN, |I| <+ \infty\}$, then $Y:= X \cup (- X)$  is a uniform approximate lattice in $\RR$, see  \cite[8.2]{meyer1972algebraic} .
\end{enumerate}

\end{Ex}

The following statement gives a handy characterisation of uniform approximate lattices.

\begin{Prop}[Proposition $2.9$, \cite{bjorklund2016approximate}]\label{Caractérisation des réseaux approximatifs uniformes}
Let $\Lambda$ be a relatively dense symmetric subset of a locally compact group $G$ containing the identity element. Then the following statements are equivalent: 
\begin{enumerate}[label=(\roman*)]
 \item $\Lambda$ is a uniform approximate lattice;
 \item $\Lambda^k:=\{\lambda_1\cdots\lambda_k | \lambda_1,\ldots,\lambda_k \in \Lambda \}$ is discrete for all $k \in \NN$:
 \item $\Lambda^6$ is discrete;
 \item For every compact subset $K \subset G$, the set $K \cap \Lambda^3$ is finite. 
\end{enumerate}
\end{Prop}


\begin{Rem}
 Without a relative density assumption, only the implications $(i)\Rightarrow(ii)\Rightarrow(iii)\Rightarrow(iv)$ hold. Indeed, for every positive integer $k$ there is a symmetric set $X \subset \RR$ such that  $0 \in X$, and $\underbrace{X+\ldots + X}_{k-\text{terms}}$ satisfies $(iv)$ whereas $\underbrace{X+\ldots + X}_{k+1-\text{terms}}$ does not. 

Set $Y:=\{k^n + \frac{1}{k^n} | n \geq 1\}$ and $X:= Y \cup (-Y)\cup \{0\}$. A quick computation shows that a sum of $k$ elements of $X$ is equal to zero or is greater than $k - 1$ in absolute value, while $\frac{1}{k^{n+1}} - \frac{1}{k^{n-1}} \in \underbrace{X+\ldots + X}_{k+1-\text{terms}}$ for all $n \geq 1$. 
\end{Rem}

\begin{Cor}[Corollary $2.10$, \cite{bjorklund2016approximate}]\label{relatively dense subset of a uniform approximate lattice}
 Let $\Lambda \subset G$ be a uniform approximate lattice in a locally compact group. A symmetric subset containing the identity of $\Lambda$ is a uniform approximate lattice if and only if it is relatively dense.
\end{Cor}

For convenience, we mention yet another result proved in $\cite{bjorklund2016approximate}$.

\begin{Prop}[Theorem $4.25$, \cite{bjorklund2016approximate}]
 Every approximate lattice (as defined in \cite[Section $4$]{bjorklund2016approximate}) in a locally compact second countable nilpotent group is a uniform approximate lattice.  
\end{Prop}

From now on, when dealing with nilpotent groups we will say `approximate lattices' instead of `uniform approximate lattices'.

\subsection{Model sets and Meyer sets}

The following scheme is the main tool to build uniform approximate lattices.

\begin{Def}
 A \emph{cut-and-project scheme} is a triple $(G,H,\Gamma)$ such that $G$ and $H$ are locally compact groups,  $\Gamma < G \times H$ is a lattice and, when restricted to $\Gamma$, the projection $\pi_G$ to $G$ is injective and the projection $\pi_H$ to $H$ has dense image. 
 Hence, we can define a group homomorphism $\tau : \pi_{G}(\Gamma) \rightarrow H$ by $$\tau := \pi_H \circ \left(\rest{\pi_G}{\Gamma}\right)^{-1}.$$ The map $\tau$ is sometimes called the \emph{star map}. 
\end{Def}

Given a cut-and-project scheme $(G,H,\Gamma)$ we are able to build a whole family of uniform approximate lattices in $G$. Indeed, choose a compact neighbourhood $W_0$ of the identity element in $H$ and define:
$$ P_0(G,H,\Gamma,W_0): = \pi_G((G \times W_0) \cap \Gamma) = \tau^{-1}(W_0). $$

\begin{Prop}[Proposition $2.13$, \cite{bjorklund2016approximate}]\label{Les ensembles modèles sont des réseaux approximatifs uniformes}
 Denote $P_0(G,H,\Gamma,W_0)$ by $P_0$. Then: 
 \begin{enumerate}
  \item$P_0$ and $P_0^{-1}P_0$ are uniformly discrete;
  \item there is a finite subset $F \subset G$ such that $P_0^2 \subset F\cdot P_0$, and thus if $W_0$ is symmetric and contains the identity, $P_0$ is an approximate subgroup;
  \item if $\Gamma$ is a uniform lattice, then $P_0$ is relatively dense;
  \item if $P_0$ is relatively dense, then $\Gamma$ is a uniform lattice.
 \end{enumerate}

\end{Prop}

\begin{Def}
 For a cut-and-project scheme $(G,H,\Gamma)$ and a compact neighbourhood of the identity $W_0 \subset H$, the set $P_0(G,H,\Gamma,W_0)$ is called a \emph{model set}. Moreover, any relatively dense subset of a model set is called a \emph{Meyer set}.
\end{Def}

\begin{Rem}
 According to Proposition \ref{Les ensembles modèles sont des réseaux approximatifs uniformes} model sets are uniform approximate lattices whenever $W_0$ is symmetric and $\Gamma$ is uniform. Moreover, under the same conditions Meyer sets are uniform approximate lattices according to Corollary \ref{relatively dense subset of a uniform approximate lattice}.
\end{Rem}

Conversely, we can state Meyer's theorem as follows.

\begin{Thm}[Theorem $IV$, \cite{meyer1972algebraic}]\label{Meyer's theorem} In compactly generated locally compact abelian groups, all approximate lattices are Meyer sets.
\end{Thm}

\section{Approximate Lattices in Nilpotent Lie Groups}

In this section we will show Theorem \ref{Lemme de rigidité} and Theorem \ref{Théorème principal}. 

Recall that we consider lattices in simply connected nilpotent Lie groups. Let $G$ be such a group and let $\mathfrak{g}$ be its Lie algebra. The exponential map $\exp:\mathfrak{g} \rightarrow G$ is then a diffeomorphism. Denote by $\log$ its inverse. The group structure on $G$ and the Lie algebra structure on $\mathfrak{g}$ are linked by the Baker-Campbell-Hausdorff formula:
$$\log(\exp X\exp Y) =
\sum_{n = 1}^\infty\frac {(-1)^{n-1}}{n}
\sum_{\begin{smallmatrix} r_1 + s_1 > 0 \\ \vdots \\ r_n + s_n > 0 \end{smallmatrix}}
\frac{[ X^{r_1} Y^{s_1} X^{r_2} Y^{s_2} \dotsm X^{r_n} Y^{s_n} ]}{\sum_{i = 1}^n (r_i + s_i) \cdot \prod_{i = 1}^n r_i! s_i!},
$$ where the sum is performed over all non-negative values of $s_i$ and $r_i$ and:

$$ [ X^{r_1} Y^{s_1} \dotsm X^{r_n} Y^{s_n} ] := [ \underbrace{X,[X,\dotsm[X}_{r_1} ,[ \underbrace{Y,[Y,\dotsm[Y}_{s_1} ,\,\dotsm\, [ \underbrace{X,[X,\dotsm[X}_{r_n} ,[ \underbrace{Y,[Y,\dotsm Y}_{s_n} ]]\dotsm]].$$

\subsection{Logarithms of  Approximate Lattices}

Thanks to the Baker-Campbell-Hausdorff formula we will be able to express the sum and Lie bracket of two elements of the Lie algebra $\mathfrak{g}$ using the logarithm map and products. 

\begin{Def}
 Let $n \in \NN$ and $w \in F_n$, an element of the free group of rank $n$ with $S=\{s_1,\ldots,s_n\}$ a set of generators. For any group $G$ and elements $g_1,\ldots,g_n \in G$ we denote by $w(g_1,\ldots,g_n)$ the image of $w$ by the only group homomorphism $f : F_n \rightarrow G$ such that $s_i \mapsto g_i$ for $i \in \{1,\ldots,n\}$. This is called a \emph{word in $n$ letters}. 
\end{Def}

As a consequence of the Baker-Campbell-Hausdorff formula, we get the following statement.

\begin{Lem}\label{Lemme technique sur l'existence d'un mot pour somme  et crochet}
 For $c \in \NN$, there exist words in two letters $w_c,w'_c$ and natural numbers $m_c,m'_c$, depending only on $c$, such that if $N$ is a simply connected nilpotent Lie group of nilpotency class $\leq c$, and $x,y \in N$ then:
 $$ m_c(\log(x) + \log(y)) = \log(w_c(x,y)) \text{ and } m'_c([\log(x),\log(y)])=\log(w'_c(x,y)).$$

\end{Lem}

A direct consequence of this lemma is the following.

\begin{Cor}\label{Corollaire de l'existence d'un mot pour la somme et le crochet}
 Let $X \subset N$ be a symmetric subset of a simply connected nilpotent Lie group of nilpotency class $c$. Then there are integers $n_c,n'_c,m_c,m'_c \in \NN$ depending only on $c$ such that 
 $$ \log(X) + \log(X) \subset \frac{1}{m_c}\log(X^{n_c}) \text{ and } [\log(X),\log(X)] \subset \frac{1}{m'_c}\log(X^{n'_c}). $$
\end{Cor}

In particular, as $\log$ is a homeomorphism, if $X^n$ is discrete for all $n \in \NN$ then $\log(X^n)$ is discrete and so $\log(X)+ \cdots + \log(X)$ and $\log(X) + [\log(X),\log(X)]$ are discrete, according to Corollary \ref{Corollaire de l'existence d'un mot pour la somme et le crochet}.

Now let us show a result about relatively dense sets in nilpotent groups.

\begin{Lem}\label{Lorgarithme des ensembles relativement denses}
 Let $X$ be a relatively dense symmetric subset of a simply connected nilpotent Lie group $N$ and suppose $e \in X$. Then there is $n_0 \in \NN$ such that for all $n \geq n_0$ the set $\log(X^n)$ is relatively dense in $\mathfrak{g}$. Moreover, $n_0$ depends only on the nilpotency class of $N$. 
\end{Lem}

\begin{proof}
 Let us prove Lemma \ref{Lorgarithme des ensembles relativement denses} by induction on the nilpotency class $c$. If $c=1$ $\log$ is also a group homomorphism so we can take $n_0=1$. 
 
 Assume the lemma to be true for a given $c$ and let $N$ be a simply connected nilpotent Lie group with nilpotency class $c+1$. Let $p:N \rightarrow N / Z(N)$ be the natural projection where $Z(N)$ is the centre of $N$ and $\pi:\mathfrak{n} \rightarrow \mathfrak{n}/ \mathfrak{z}(\mathfrak{n})$ be the natural projection where $\mathfrak{z}(\mathfrak{n})$ is the centre of $\mathfrak{n}$ as a Lie algebra. Then $p(X)$ is relatively dense in $N/Z(N)$ and by the induction hypothesis there is $n_0 \in \NN$ such that for all $n \geq n_0$ we have $\log(p(X)^n)=\pi(\log(X^n))$ is relatively dense too. Thus, there is a compact subset $K \subset \mathfrak{n}$ such that $\mathfrak{n} = \log(X^n) + K + \mathfrak{z}(\mathfrak{n})$. 
 
 Let $ L \subset N$ be a compact subset such that $X^n\cdot L = N$. Then for any $z \in \mathfrak{z}(\mathfrak{n})$ there are $x \in X^n$ and $b \in L$ such that $\exp(z)=x b$. Thus, $x$ and $b$ commute so $z = \log(x) + \log(b)$. Hence, $\mathfrak{z}(\mathfrak{n}) \subset \log(X^n) + \log(L)$ so $$\log(X^n) + K + \log(X^n) + \log(L)= \mathfrak{n}.$$ But according to Corollary \ref{Corollaire de l'existence d'un mot pour la somme et le crochet} there are $m_{c+1},n_{c+1} \in \NN$ such that $$\log(X^n)+\log(X^n) \subset \frac{1}{m_{c+1}}\log(X^{n_{c+1}n}),$$
 so 
 $$\log(X^{n_{c+1}n}) + m_{c+1}(K+\log(L)) = \mathfrak{n}.$$
 Therefore, for all $n \geq n_{c+1}n_0$ the set $\log(X^n)$ is relatively dense.
 \end{proof}

 \begin{Rem}
  According to the proof of Lemma \ref{Lorgarithme des ensembles relativement denses}, if $N$ has nilpotency class $c$ then we can take $n_0$ to be $n_1n_2\cdots n_c$ where $n_c$ is the integer given by Corollary \ref{Corollaire de l'existence d'un mot pour la somme et le crochet}.
 \end{Rem}

As a consequence of the lemmas stated above we obtain:

\begin{Prop}\label{Consequence logarithme des ensembles relativement denses}
 Let $\Lambda \subset N$ be an approximate lattice in a simply connected nilpotent Lie group. Then, there is a natural number $n_0$, depending only on the nilpotency class of $N$, such that for all $n \geq n_0$ the sets $\log(\Lambda^{n})$ and $\log(\Lambda^{n})+[\log(\Lambda^{n}),\log(\Lambda^{n})]$ are approximate lattices in the Lie algebra $\mathfrak{n}$ with its additive group structure. 
\end{Prop}

\begin{proof}
 First of all, for all $n \in \NN$ both $\log(\Lambda^{n})$ and $\log(\Lambda^{n})+[\log(\Lambda^n),\log(\Lambda^n)]$ are symmetric subsets containing $0$. Now, according to Lemma \ref{Lorgarithme des ensembles relativement denses} there is $n_0 \in \NN$ such that for any integer $n\geq n_0$ the set $\log(\Lambda^n)$ is relatively dense and so is $\log(\Lambda^{n})+[\log(\Lambda^n),\log(\Lambda^n)]$. As a consequence of Corollary \ref{Corollaire de l'existence d'un mot pour la somme et le crochet} there is $m \in \NN$ such that 
 $$ \left(\log(\Lambda^{n})+[\log(\Lambda^{n}),\log(\Lambda^{n})]\right) +\cdots+ \left( \log(\Lambda^{n})+[\log(\Lambda^{n}),\log(\Lambda^{n})]\right) \subset \log(\Lambda^m),$$
 hence is discrete. Similarly, 
 $$ \log(\Lambda^{n})+\cdots+ \log(\Lambda^{n}) \subset \log(\Lambda^m),$$
 so is a discrete set. 
 
 So for all $n \geq n_0$ the sets $\log(\Lambda^{n})$ and $\log(\Lambda^{n})+[\log(\Lambda^{n}),\log(\Lambda^{n})]$ are approximate lattices according to Proposition \ref{Caractérisation des réseaux approximatifs uniformes}.
\end{proof}

\subsection{Rigidity}

Now, let us turn to the proof of Theorem \ref{Lemme de rigidité}. Recall its statement.

\begin{Thm*}
 Let $\Lambda \subset G$ be an approximate lattice in a simply connected nilpotent Lie group and $\Gamma: = \langle \Lambda \rangle$ be the subgroup generated by $\Lambda$. Let $f: \Gamma \rightarrow N$ be an abstract group homomorphism from $\Gamma$ to a simply connected nilpotent Lie group. Then there are unique group homomorphisms $\tilde{f}: G \rightarrow N$ and $\rho: \Gamma \rightarrow N$ such that :
\begin{enumerate}[label=(\roman*)]
 \item $\tilde{f}$ is a Lie group homomorphism;
 \item $f=\rest{\tilde{f}}{\Gamma} \cdot \rho$;
 \item The image of $\rho$ lies in the centraliser $C(\im(\tilde{f}))$ of $\im(\tilde{f})$;
 \item $\rest{\rho}{\Lambda}$ is bounded.
\end{enumerate}
\end{Thm*}

We will need an abelian version of this result.

\begin{Lem}\label{Lemme de rigidité abélien}
 Let $\Lambda \subset \RR^n$ be an approximate lattice and write $\Gamma := \langle \Lambda \rangle $. Let $\phi: \Gamma \rightarrow \RR^m$ be a group homomorphism. Then there is a continuous homomorphism $\tilde{\phi}: \RR^n \rightarrow \RR^m$ such that $\rest{(\tilde{\phi}-\phi)}{\Lambda}$ is bounded. 
\end{Lem}

A proof of this lemma can be found in \cite[Proposition $8.6$]{moody1997meyer}

\begin{proof}[Proof of existence]
 Let us start with a lemma.
 
\begin{Lem}\label{Extension à l'anneau de Lie engendré}
Denote by $A$ the Lie ring generated by $\log(\Lambda)$. Then $\log \circ f \circ \exp$ extends to a Lie ring homomorphism $\phi: A \rightarrow \mathfrak{n}$.
\end{Lem}

\begin{proof}
 According to Corollary \ref{Corollaire de l'existence d'un mot pour la somme et le crochet} we have $ A \subset \QQ(\log(\Gamma))$.

 Assume $G$ and $N$ are of nilpotency class less than or equal to $c$ and $w_c, m_c$ are as in Lemma \ref{Lemme technique sur l'existence d'un mot pour somme  et crochet}. Define by induction the word $w_{c,n}$ for $n \geq 2$:
 \begin{align*}
  & w_{c,2}=w_c ;\\
  & \forall n \geq 2, w_{c,n+1}(x_1,\ldots,x_{n+1}) = w_c(w_{c,n}(x_1,\ldots,x_n), x_{n+1}^{m_c^{n-1}}).
 \end{align*}
 Then, for $g_1,\ldots,g_n \in G$, we have $$\log(w_{c,n}(g_1,\ldots,g_n)) = m_c^{n-1}(\log(g_1) +\cdots+\log(g_n)),$$
 and the same is true for elements of $N$. Therefore, for $\gamma_1,\ldots,\gamma_n, \gamma'_1,\ldots,\gamma'_n \in \Gamma$ such that 
 $$\log(\gamma_1)+\cdots + \log(\gamma_n) = \log(\gamma'_1)+\cdots + \log(\gamma'_n)$$
 we have 
 $$w_{c,n}(\gamma_1,\ldots,\gamma_n)=w_{c,n}(\gamma'_1,\ldots,\gamma'_n).$$
 Then 
 $$w_{c,n}(f(\gamma_1),\ldots,f(\gamma_n))=w_{c,n}(f(\gamma'_1),\ldots,f(\gamma'_n))$$
 because $f$ is a group homomorphism. Hence, 
 $$\phi(\log(\gamma_1))+\cdots + \phi(\log(\gamma_n)) = \phi(\log(\gamma'_1))+\cdots + \phi(\log(\gamma'_n)),$$
 where $\phi$ denotes $\log \circ f \circ \exp$. So we can extend $\log\circ f\circ \exp$ to a group homomorphism $A \rightarrow \mathfrak{n}$, again denoted by $\phi$.
 
 It remains to prove that $\phi$ is a Lie ring homomorphism. For any $x_1,x_2 \in A$ there is a natural number $m$ such that $mx_1,mx_2 \in \log(\Gamma)$. Let $\gamma_1,\gamma_2 \in \Gamma$ be such that $\log(\gamma_1)=mx_1$ and $\log(\gamma_2)=mx_2$, and let $w'_c$ and $m'_c$ be as in Lemma \ref{Lemme technique sur l'existence d'un mot pour somme  et crochet}. Then 
 \begin{align*}
 \phi([x_1,x_2])= & \frac{1}{m'_cm^2}\phi(m'_c[\log(\gamma_1),\log(\gamma_2)])\\
		= &\frac{1}{m'_cm^2}\phi(\log(w'_c(\gamma_1,\gamma_2))) \\
		= &\frac{1}{m'_cm^2}\log(w'_c(f(\gamma_1),f(\gamma_2))) \\
		= &\frac{1}{m^2}[\log(f(\gamma_1)),\log(f(\gamma_2))]= [\phi(x_1),\phi(x_2)].
 \end{align*}
\end{proof}

According to Proposition \ref{Consequence logarithme des ensembles relativement denses}, up to considering $\Lambda^k$ for $k$ large enough, we can assume that $\log(\Lambda)$ and $\log(\Lambda) + [\log(\Lambda),\log(\Lambda)]$ are approximate lattices. Denote by $\phi$ the Lie ring homomorphism given by Lemma \ref{Extension à l'anneau de Lie engendré}, and $\tilde{\phi}$ the one given by Lemma \ref{Lemme de rigidité abélien} applied to $\phi$ and $\log(\Lambda) + [\log(\Lambda),\log(\Lambda)]$.

Now, let us show that $\tilde{\phi}$ is a Lie algebra homomorphism. Set

\begin{equation*}
 B:=\overline{(\phi-\tilde{\phi})(\log(\Lambda) + [\log(\Lambda),\log(\Lambda)])}.
\end{equation*}

Let $x_1,x_2 \in \log(\Lambda)$ and $N$ be a norm on $\mathfrak{n}$. There are $b \geq 0$ and $C \geq 0$ such that $B \subset B_{N}(0,b)$ and for all $y_2 \in \mathfrak{n},y_1 \in B$ :
$$ N([y_1,y_2]) \leq C N(y_2) $$
Therefore, 
\begin{align*}
N(\tilde{\phi}([x_1,x_2])-[\tilde{\phi}(x_1),\tilde{\phi}(x_2)]) & \leq b + N(\phi([x_1,x_2])-[\phi(x_1),\phi(x_2)] - [\tilde{\phi}(x_1)-\phi(x_1),\phi(x_2)] - \\
	& [\phi(x_1),\tilde{\phi}(x_2)-\phi(x_2)] - [\tilde{\phi}(x_1)-\phi(x_1),\tilde{\phi}(x_2)-\phi(x_2)])\\
	& \leq b + C(N(\phi(x_1))+N(\phi(x_2)))+ Cb \\
	& \leq b + C(N(\tilde{\phi}(x_1))+N(\tilde{\phi}(x_2)))+ 3Cb.
\end{align*}
Since $\log(\Lambda)$ is an approximate lattice in $\mathfrak{g}$ and $\tilde{\phi}$ is a linear map, we have
$$\tilde{\phi}([x_1,x_2])-[\tilde{\phi}(x_1),\tilde{\phi}(x_2)]= \mathcal{O}(N(x_1)+ N(x_2)).$$
So the bilinear form 
\begin{align*}
 \mathfrak{g}\times \mathfrak{g} \ \ & \longrightarrow \quad \mathfrak{n} \\
 (x_1,x_2)& \mapsto \tilde{\phi}([x_1,x_2])-[\tilde{\phi}(x_1),\tilde{\phi}(x_2)]
\end{align*}
is null, and hence $\tilde{\phi}$ is a Lie algebra homomorphism.

Finally, $\im(\tilde{\phi})$ and $\im(\phi-\tilde{\phi})$ commute in $\mathfrak{n}$. Indeed, choose $x_1,x_2 \in \log(\Lambda)$ and let $r$ denote $\phi - \tilde{\phi}$ then:
\begin{align*}
r([x_1,x_2]) & = \phi([x_1,x_2])-\tilde{\phi}([x_1,x_2]) \\
		&= [\phi(x_1),\phi(x_2)]- [\tilde{\phi}(x_1),\phi(x_2)] + [\tilde{\phi}(x_1),\phi(x_2)] - [\tilde{\phi}(x_1),\tilde{\phi}(x_2)]\\ 
		&=[r(x_1),\phi(x_2)]+[\tilde{\phi}(x_1),r(x_2)]
\end{align*}

So $[\tilde{\phi}(x_1),r(x_2)]= r([x_1,x_2])- [r(x_1),\phi(x_2)]$. As a consequence, for any given $x_2 \in \log(\Lambda)$ the linear map, 
\begin{align*}
 \mathfrak{g}& \rightarrow \mathfrak{n} \\
 x_1 &\mapsto [\tilde{\phi}(x_1),r(x_2)]
\end{align*}
is bounded over $\log(\Lambda)$. But $\log(\Lambda)$ is relatively dense in $\mathfrak{g}$. So $ [\tilde{\phi}(\cdot),r(x_2)]=0$. Hence, $\im(r)$ and $\im(\tilde{\phi})$ commute. 

Let us check that $\tilde{f}:=\exp \circ \tilde{\phi} \circ \log$ and $\rho := \exp \circ r \circ \log$ satisfy all conditions of Theorem \ref{Lemme de rigidité}. Condition $(iv)$ is satisfied since $r$ is bounded on $\log(\Lambda)$. The map $\tilde{\phi}$ is a homomorphism of Lie algebras, so for $x,y \in \mathfrak{g}$ we have,
\begin{align*}
\log\left(\tilde{f}\left(\exp x\exp y\right)\right) &=\tilde{\phi}\left(\sum_{n = 1}^\infty\sum_{\begin{smallmatrix} r,s \in (\ZZ_{\geq 0})^n  \\ r_i + s_i > 0 \end{smallmatrix}}\alpha_{r,s}[ x^{r_1} y^{s_1} x^{r_2} y^{s_2} \dotsm x^{r_n} y^{s_n} ]\right) \\
                                                      &=\sum_{n = 1}^\infty\sum_{\begin{smallmatrix} r,s \in (\ZZ_{\geq 0})^n  \\ r_i + s_i > 0 \end{smallmatrix}}\alpha_{r,s}[ \tilde{\phi}(x)^{r_1} \tilde{\phi}(y)^{s_1} \dotsm \tilde{\phi}(x)^{r_n} \tilde{\phi}(y)^{s_n} ] \\
                                                     & = \log\left(\exp \tilde{\phi}(x)\exp \tilde{\phi}(y)\right)\\
                                                     & = \log\left(\tilde{f}\left(\exp x\right)\tilde{f}\left(\exp y\right)\right)                                                     
\end{align*}
where the $\alpha_{r,s}$'s are the coefficients of the Baker-Campbell-Hausdorff formula. So $\tilde{f}$ is a Lie group homomorphism and $(i)$ is satisfied. 

Now, let $x,y \in \mathfrak{n}$ be such that $[x,y]=0$. Then according to the Baker-Campbell-Hausdorff formula once again, we know that 
$$\exp(x)\exp(y)=\exp(y)\exp(x)=\exp(x+y).$$
Thus, the image of $\rho$ lies in the centraliser $C(\im(\tilde{f}))$ of $\im(\tilde{f})$ since all elements of $\im(r)$ commute with all elements of $\im(\tilde{\phi})$.

Finally, for all $\gamma \in \Gamma$, $\phi(\log\left(\gamma\right))=\tilde{\phi}(\log\left(\gamma\right)) + r(\log\left(\gamma\right))$. But $\tilde{\phi}(\log\left(\gamma\right))$ and $r(\log\left(\gamma\right))$ commute so 
$$ f(\gamma) = \exp\left(\tilde{\phi}(\log\left(\gamma\right)) + r(\log\left(\gamma\right))\right) = \tilde{f}(\gamma)\rho(\gamma).$$
 \end{proof}
\begin{proof}[Proof of uniqueness]

\begin{Lem}\label{Egalité de deux morphismes à distance bornée sur un réseau approximatif}
Let $f,g:G \rightarrow N$ be Lie group homomorphisms. If $\rest{f(\cdot)g(\cdot)^{-1}}{\Lambda}$ is bounded, then $f=g$.
\end{Lem}

Pick some compact subset $K \subset G$ such that $K\Lambda=G$. Then, for all $x \in G$ there are $\lambda \in \Lambda$ and $b \in K$ such that $x=b\lambda$, so $$f(x)g(x)^{-1}=f(b)f(\lambda)g(\lambda)^{-1}g(b)^{-1}.$$
Therefore, the map $g \mapsto f(g)g(g)^{-1}$ is bounded.

We proceed by induction on $n$ the dimension of $N$ as a Lie group. If $n=1$ the claim is true. Now, assume that the induction hypothesis is true up to a given $n \in \NN$. Let $\pi:N \rightarrow N/Z(N)$ be the canonical projection, then $\pi\circ f=\pi \circ g$ by the induction hypothesis. So $g(g)^{-1}f(g) \in Z(N)$ for all $g \in G$. Therefore for $g_1,g_2 \in G$ we have 
\begin{align*}
 g(g_1g_2)^{-1}f(g_1g_2)& = g(g_2)^{-1}g(g_1)^{-1}f(g_1)f(g_2) \\
			& = g(g_2)^{-1}f(g_2)g(g_1)^{-1}f(g_1).
\end{align*}
So the map $g \mapsto g(g)^{-1}f(g)$ is a group homomorphism. In addition, it has bounded image and target $Z(N)\simeq \RR^m$. As a conclusion, $x \mapsto g(x)^{-1}f(x)$ is the trivial homomorphism. 
\end{proof}

Note that if $\Lambda$ is a uniform lattice then Theorem \ref{Lemme de rigidité} becomes as follows.

\begin{Thm}[\cite{mal1949class}]\label{Rigidité de Malcev}
 Let $\Lambda \subset G$ be a uniform lattice in a simply connected nilpotent Lie group, and $f : \Lambda \rightarrow N$ a group homomorphism with target a simply connected nilpotent Lie group. Then, $f$ extends to a unique Lie group homomorphism $\tilde{f}: G \rightarrow N$. 
\end{Thm}

This is the well known Malcev rigidity Lemma, for a proof see \cite[Theorem $2.11$]{raghunathan1972discrete}.

\subsection{Proof of Theorem \ref{Théorème principal}}


The proof of Theorem \ref{Théorème principal} is similar to a proof of Theorem \ref{Meyer's theorem} due to Schreiber (see \cite[Theorem 2]{schreiber1973approximations}). The proof has four steps:
\begin{enumerate}[label=(\roman*)]
 \item we first embed the group $\langle \Lambda \rangle$ in some simply connected group $N$;
 \item then we use the rigidity of lattices to obtain a homomorphism $\pi:N \rightarrow G$;
 \item thanks to the rigidity of approximate lattices we find a section to $\pi$;
 \item finally, with this section we show that $N \simeq G \times H$ for some group $H$ and we get the desired cut-and-project scheme.
\end{enumerate}
In the case of Theorem \ref{Meyer's theorem}, there is $k \in \NN$ such that the groups $\langle \Lambda \rangle$ and $N$ are isomorphic to $\ZZ^k$ and $\RR^k$ respectively. Therefore, steps $(i)$,$(ii)$ and $(iv)$ are elementary while step $(iii)$ is a consequence of Lemma \ref{Lemme de rigidité abélien}. In the nilpotent case, we have to rely on a number of theorems due to Malcev concerning lattices in nilpotent groups to deal with steps $(i)$ and $(ii)$. Step $(iii)$ is a consequence of Theorem \ref{Lemme de rigidité}. Moreover, step $(iv)$ also becomes more intricate when $N$ is not a vector space. Indeed, in the case of Theorem \ref{Meyer's theorem} step $(iv)$ essentially reduces to the following statement that fails to be true for nilpotent Lie groups : every subspace of a vector space has a complement.
\begin{proof}[Proof of Theorem \ref{Théorème principal}.]
 The subgroup $\langle \Lambda \rangle$ generated by $\Lambda$ is finitely generated \cite[Theorem $3.4$]{bjorklund2016approximate}, nilpotent and torsion-free. According to a theorem of Malcev it is therefore isomorphic as an abstract group to a lattice $\Gamma$ in a simply connected nilpotent Lie group $N$ (see \cite[Theorem $2.18$]{raghunathan1972discrete}). Let $i: \langle \Lambda \rangle \rightarrow \Gamma$ be an isomorphism.
 
 According to Malcev rigidity (Theorem \ref{Rigidité de Malcev}) there is a unique Lie group homomorphism $\pi: N \rightarrow G$ extending $i^{-1}$. The morphism $\pi$ has connected co-compact image in $G$, so $\pi$ is onto. 
 
 Now, let $\tilde{i}$ and $\rho$ be given by Theorem \ref{Lemme de rigidité} applied to $i$. First of all, let us prove that $\tilde{i}$ is a section of $\pi$. Indeed, both $\pi \circ \tilde{i}$ and $\Id_G$ satisfy the conditions $(i)-(iv)$ of Theorem \ref{Lemme de rigidité} applied to $\pi \circ i$, so $\pi \circ \tilde{i}= \Id_G$. In particular, $\im(\rho) \subset \ker(\pi)$.
 
 Define $H:= \ker(\pi)$ and let us show that $N = H\times \tilde{i}(G)$. The group homomorphism $\rho \circ i^{-1}$ extends to a unique morphism $\tilde{\rho}: N \rightarrow N$ according to Theorem \ref{Rigidité de Malcev}. Moreover, by construction $\im(\rho)$ is co-compact in $\im(\tilde{\rho})$. Since $\im(\tilde{\rho})$ and $C(\im(\tilde{i}))$ are connected the subgroup $\im(\tilde{\rho}) \cap C(\im(\tilde{i}))$ is connected (\cite[Lemma $2.4$]{raghunathan1972discrete}). In addition, since $\im(\rho) \subset C(\im(\tilde{i}))$ the subgroup $\im(\tilde{\rho}) \cap C(\im(\tilde{i}))$ contains $\im(\rho)$ as a co-compact subgroup. Hence, $$\im(\tilde{\rho}) \cap C(\im(\tilde{i}))= \im(\tilde{\rho})$$ according to \cite[Theorem $2.1$]{raghunathan1972discrete}. So the inclusion $\im(\tilde{\rho}) \subset C(\im(\tilde{i}))$ holds. Similarly one can prove that $\im(\tilde{\rho}) \subset H$.
 
 Now, the map
 \begin{align*}
   & N  \rightarrow N \\
& n  \mapsto \tilde{\rho}(n)(\tilde{i}\circ\pi)(n)
 \end{align*}
 is a group homomorphism that induces the identity on $\Gamma$. So for all $n \in N$, $\tilde{\rho}(n)(\tilde{i}\circ\pi)(n)=n$ according to Theorem \ref{Rigidité de Malcev}. 
 
 Now, we see that $\im(\tilde{\rho}) = H$ and $N = H \times \tilde{i}(G)$. Furthermore, $\pi_H = \tilde{\rho}$ and $\pi_{\tilde{i}(G)}=p$. 
 
 Finally, $\rho(\Lambda)$ is bounded so any compact neighbourhood $W_0$ of $e$ containing $\rho(\Lambda)$ works. 
 
\end{proof}

\subsection{From Theorem \ref{Théorème principal} to Corollary \ref{Théorème de Meyer pour les groupes de Lie nilpotents}}

Recall the statement of Corollary \ref{Théorème de Meyer pour les groupes de Lie nilpotents}. 

\begin{Cor*}
 Let $\Lambda \subset G$ be an approximate lattice in a connected nilpotent Lie group $G$. Then there exists a simply connected nilpotent Lie group $H$ such that $\Lambda$ is a Meyer set given by a cut-and-project scheme $(G,H,\Gamma)$. 
\end{Cor*}

The proof relies on the following general result. 

\begin{Prop}\label{Réseaux approximatifs uniformes et revêtements}
Let $p:G \rightarrow B$ be a covering of locally compact groups, and $\Lambda \subset B$ a uniform approximate lattice. Then
\begin{enumerate}
 \item $p^{-1}(\Lambda)$ is a uniform approximate lattice in $G$;
 \item if $p^{-1}(\Lambda)$ is a Meyer set with respect to a cut-and-project scheme $(G,H,\Gamma)$ such that $H$ has no non-trivial compact subgroup, then $\Lambda$ is a Meyer set with respect to the cut-and-project scheme $(B,H,\Gamma/(\ker(p)\times \{e\}))$. 
\end{enumerate}

\end{Prop}

\begin{proof}
\begin{enumerate}
\item First of all, $p^{-1}(\Lambda)$ is symmetric. For all $n \in \NN$, the set $\Lambda^n$ is discrete. As $p$ is a local homeomorphism $p^{-1}(\Lambda)^n$ is also discrete. According to Proposition \ref{Caractérisation des réseaux approximatifs uniformes} it suffices to prove that $p^{-1}(\Lambda)$ is relatively dense. Let $K \subset B$ be a compact set such that $\Lambda K= B$. As $p$ is a covering and $G$ is locally compact there is a compact subset $L \subset G$ such that $K \subset p(L)$. Now, for all $x \in G$ there is $\lambda \in \Lambda$ and $k \in K$ such that $p(x)=\lambda k$. Let $l \in L$ be such that $p(l)=k$, then $p(xl^{-1}) = \lambda \in \Lambda$ so $p^{-1}(\Lambda) L = G$.

\item Pick some compact set $W_0\subset G$ such that $p^{-1}(\Lambda)\subset\pi_G(\left(G\times W_0\right)\cap \Gamma )$. As $\ker(p)\subset p^{-1}(\Lambda)$ we have $\overline{\tau(\ker(p))} \subset W_0$, so $\overline{\tau(\ker(p))}$ is compact. By assumption,  $\tau(\ker(p))=\{e\}$ which implies $$\ker(p\times \Id_H)=\ker(p)\times \{e\} \subset \Gamma.$$
We then get 
$$\Lambda\subset\pi_G(\left(B\times W_0\right)\cap \left(p\times\Id_H(\Gamma)\right) ).$$
But $p$ is a covering map so $p\times\Id_H : G \times H \rightarrow B \times H$ is a surjective continuous homomorphism and is a local homeomorphism. Since $p\times\Id_H$ is a local homeomorphism the subgroup $p\times\Id_H(\Gamma)$ is discrete. Moreover, $p\times\Id_H(\Gamma)$ is co-compact as $p\times\Id_H$ is a surjective continuous homomorphism and $\Gamma$ is co-compact. So $p\times\Id_H(\Gamma)$ is a uniform lattice in $B\times H$.
\end{enumerate}
\end{proof}

\begin{proof}[Proof of Corollary \ref{Théorème de Meyer pour les groupes de Lie nilpotents}.]
 Let $N$ be the universal cover of the connected group $G$ and $p:N \rightarrow G$ be the associated covering map. The group $N$ is a simply connected nilpotent Lie group, and $p^{-1}(\Lambda) \subset N$ is an approximate lattice according to Part $1$ of Proposition \ref{Réseaux approximatifs uniformes et revêtements}. So, by Theorem \ref{Théorème principal}, the approximate lattice $p^{-1}(\Lambda)$ is a Meyer set with respect to a cut-and-project scheme $(N,H,\Gamma)$ where $H$ is a simply connected nilpotent Lie group. In particular, $H$ has no non-trivial subgroup. According to Part $2$ of Proposition \ref{Réseaux approximatifs uniformes et revêtements} the approximate subgroup is a Meyer set with respect to the cut-and-project scheme $(G,H,\Gamma/(\ker(p)\times \{e\}))$.
\end{proof}

As an easy consequence we show the following corollary. 

\begin{Cor}\label{Intersection uniform approximate lattice and derived subgroup}
 Let $\Lambda \subset G$ be an approximate lattice in a nilpotent Lie group and $p: G \rightarrow G/[G;G]$ the canonical projection. Then,
 \begin{enumerate}
  \item the set $p(\Lambda)$ is an approximate lattice in $G/[G,G]$;
  \item $\Lambda^2 \cap [G,G]$ is an approximate lattice in $[G,G]$.
 \end{enumerate}
\end{Cor}

\begin{proof}
According to Corollary \ref{Théorème de Meyer pour les groupes de Lie nilpotents} there are a nilpotent Lie group $H$, a lattice $\Gamma < G \times H$ and a symmetric compact neighbourhood $W_0$ of $e$ in $H$ such that $\Lambda \subset \pi_G\left(\left(G\times W_0\right) \cap \Gamma\right)$.
 \begin{enumerate} 
 
 \item As $\Lambda$ is a relatively dense approximate subgroup, so is $p(\Lambda)$. Let $q: H \rightarrow H/[H,H]$ be the natural projection. We have
  $$p(\Lambda) \subset \pi_{G/[G,G]}((G/[G,G] \times q(W_0)) \cap (p \times q)(\Gamma))$$ where $\pi_{G/[G,G]}$ is the projection to the first factor. We know that $\Gamma \cap [G,G]\times [H,H]$ is a uniform lattice in $[G,G]\times [H,H]$ (\cite[Corollary $1$ to Theorem $2.3$]{raghunathan1972discrete}). By \cite[Theorem $1.13$]{raghunathan1972discrete} this implies that $\left(p \times q \right)(\Gamma)$ is a uniform lattice in $G/[G,G] \times H/[H,H]$. Thus, $p(\Lambda)$ is uniformly discrete according to Part $1$ of Proposition \ref{Les ensembles modèles sont des réseaux approximatifs uniformes}. The set $p(\Lambda)$ is therefore an approximate lattice according to Proposition \ref{Caractérisation des réseaux approximatifs uniformes}. 
  
  \item Since $p(\Lambda)$ is an approximate lattice, the set $\Lambda^2 \cap [G,G]$ is an approximate lattice in $[G,G]$ according to \cite[Theorem $3.3$]{bjorklund2018spectral}.
 
 \end{enumerate}
\end{proof}

\begin{Rem}
 According to \cite[Lemma $10.3$]{MR3090256}, for any locally compact group $G$, any uniform approximate lattice $\Lambda \subset G$ and any subgroup $H \subset G$, the set $\Lambda^2 \cap H$ is a uniformly discrete approximate subgroup. It is therefore natural to ask whether $\Lambda^2 \cap H$ is a uniform approximate lattice in $H$. Corollary \ref{Intersection uniform approximate lattice and derived subgroup} answers this question positively when $G$ is a simply connected nilpotent Lie group and $H$ is the derived subgroup. It turns out that similar results hold when $G$ is a Lie group and $H$ is the centraliser of $\Lambda$, or under reasonable hypotheses when $G$ is a Lie group and $H$ is its largest connected nilpotent subgroup (see \cite{machado2019goodmodels}).
\end{Rem}

\section{Criterion for Existence}

%
%
%
%
%
%
%

In this section we prove Theorem \ref{Existence d'un réseau approximatif}. Let us first recall the definition of structure constants. 

\begin{Def}
 Let $\mathfrak{g}$ be a Lie algebra over a field $K$, and $(e_i)$ a basis of $\mathfrak{g}$ as a $K$-vector space. The coordinates in the basis $(e_i)$ of the elements $([e_i,e_j])_{i,j}$ are called \emph{structure constants}. 
\end{Def} 

To prove Theorem \ref{Existence d'un réseau approximatif} we will need to work with Lie algebras over varying ground fields.

\begin{Def}[Chapter 1, \S 1.9, \cite{MR0132805}]
 Let $K\hookrightarrow L  $ be a field extension and $\mathfrak{g}$ a $K$-Lie algebra. Then the $L$-vector space $\mathfrak{g}\otimes L$ can be endowed with a $L$-Lie algebra structure by extending linearly the Lie bracket of $\mathfrak{g}$. We will denote by $\mathfrak{g}(L)$ the $L$-Lie algebra obtained this way. 
\end{Def}

\begin{Rem}\label{Remarque algebres de Lie}
This operation is well behaved. For instance we have $(\mathfrak{g}\oplus\mathfrak{h})(L)=\mathfrak{g}(L)\oplus\mathfrak{h}(L)$ (\cite[Chapter 2, \S 3.7, Corollary 3]{MR0274237}) and for any extension $ L \hookrightarrow M$ we have $(\mathfrak{g}(L))(M)=\mathfrak{g}(M)$ (\cite[Chapter 2, \S 5.1, Proposition 2]{MR0274237}). Moreover, a $L$-Lie algebra $\mathfrak{g}$ has a basis with structure constants lying in a subfield $K\subset L$ if and only if there is a $K$-Lie algebra $\mathfrak{h}$ such that $\mathfrak{h}(L)\simeq \mathfrak{g}$ (\cite[Chapter 1, \S 1.9]{MR0132805}). Finally, we will often identify $\mathfrak{g}$ with its image by the map 
 \begin{align*}
  \mathfrak{g} &\rightarrow \mathfrak{g}(L) \\
  x &\mapsto x \otimes 1.
 \end{align*}
\end{Rem}

If a simply connected nilpotent Lie group $G$ contains an approximate lattice, then there is a simply connected nilpotent Lie group $H$ such that $G \times H$ contains a uniform lattice, according to Theorem \ref{Théorème principal}. Thus, $\mathfrak{g}\oplus \mathfrak{h}$ has a basis with rational structure constants by Theorem \ref{CNS existence d'un réseau uniforme}. 

\begin{Def}
 Let $\mathfrak{g}$ be a non-trivial $K$-Lie algebra, then $\mathfrak{g}$ is \emph{indecomposable} if there are no non-trivial $K$-Lie algebras $\mathfrak{g_1},\mathfrak{g_2}$ such that $\mathfrak{g_1} \oplus \mathfrak{g_2} \simeq \mathfrak{g}$. 
\end{Def}

In fact, any Lie algebra can be written in an essentially unique manner as a product of indecomposable Lie subalgebras.

\begin{Prop}[Theorem $3.3$, \cite{fisher2013automorphisms}]\label{Decomposition en indecomposables}
 Let $\mathfrak{g}$ be a $K$-Lie algebra, then there are indecomposable $K$-Lie subalgebras $I_1,\ldots,I_r\subset \mathfrak{g}$ such that:
 $$ \mathfrak{g} = I_1\oplus \cdots \oplus I_r.$$
 Moreover, if $J_1,\ldots,J_s\subset \mathfrak{g}$ are other indecomposable $K$-Lie subalgebras such that $ \mathfrak{g} = J_1 \oplus \cdots \oplus J_s $ then $s=r$ and there is a bijection $\sigma:\{1,\ldots,r\} \rightarrow \{1,...,r\}$ such that $\forall i \in \left\{1,\ldots,r\right\}$:
 $$ I_i \simeq J_{\sigma(i)}.$$ 
\end{Prop}

\begin{Prop}\label{Algèbres absolument indécomposables}
 Let $\mathfrak{g}$ be a $\RR \cap \overline{\QQ}$-Lie algebra. Then $\mathfrak{g}(\RR)$ is indecomposable if and only if $\mathfrak{g}$ is indecomposable. 
\end{Prop}

\begin{proof}

Let $(e_1,\ldots,e_n)$ be a $\RR \cap \overline{\QQ}$-basis of $\mathfrak{g}$. For any extension $K$ of $\RR \cap \overline{\QQ}$, the $K$-Lie algebra $\mathfrak{g}(K)$ is indecomposable if and only if there are two matrices $A,B \in \mathcal{M}_{n,n}(K)$ such that:
 \begin{equation}\label{equation indecomposable}
\begin{cases}
& A + B = \Id_{\mathfrak{g}(K)} \\
& \det(A)=\det(B)=0\\
& A^2=A \text{ and } B^2=B \\
& A([e_i,e_j]) = [Ae_i,Ae_j], \forall i,j \in \{1,\ldots,n\} \\
& B([e_i,e_j]) = [Ae_i,Be_j], \forall i,j \in \{1,\ldots,n\}
\end{cases}
\end{equation}
Since $(e_1,\ldots,e_n)$ has structure constants in $\RR \cap \overline{\QQ}$ this system is a system of polynomial equations with coefficients in $\RR \cap \overline{\QQ}$. Moreover, $\RR \cap \overline{\QQ} $ and $\RR$ are two real closed fields (\cite[Definition $1.2.1$]{MR1659509}) so (\ref{equation indecomposable}) has a solution in $\RR \cap \overline{\QQ}$ if and only if (\ref{equation indecomposable}) has a solution in $\RR$ according to \cite[Proposition $4.1.1$]{MR1659509}.
\end{proof}

\begin{proof}[Proof of Theorem \ref{Existence d'un réseau approximatif}, necessity.]
 Let $G$ be a simply connected nilpotent Lie group containing an approximate lattice and $\mathfrak{g}$ be its Lie algebra. According to Theorems \ref{Théorème principal} and \ref{CNS existence d'un réseau uniforme} there is a Lie algebra $\mathfrak{h}$ and $(e_1,\ldots,e_n)$ a basis of $\mathfrak{g} \oplus \mathfrak{h}$ such that $(e_1,\ldots,e_n)$ has structure constants in $\QQ$. The $\QQ$-algebra $\mathfrak{k}$ generated by $(e_1,\ldots,e_n)$ is such that $\mathfrak{k}(\RR) \simeq \mathfrak{g} \oplus \mathfrak{h}$. Let $\mathfrak{k}_1,...,\mathfrak{k}_r \subset \mathfrak{k}(\RR\cap\overline{\QQ}) $ be indecomposable $\left(\RR\cap\overline{\QQ}\right)$-Lie subalgebras given by Proposition \ref{Decomposition en indecomposables} such that $$\mathfrak{k}(\RR\cap\overline{\QQ}) = \mathfrak{k}_1\oplus \cdots \oplus \mathfrak{k}_r.$$ 
 By remark \ref{Remarque algebres de Lie},
 $$ \mathfrak{k}_1(\RR)\oplus \cdots \oplus \mathfrak{k}_r(\RR) = \mathfrak{k}(\RR).$$
 Now, there is a set $I \subset \{1,\ldots, r \}$ such that $\mathfrak{g}\simeq \bigoplus\limits_{i \in I} \mathfrak{k}_i\left(\RR\right) $. Indeed, write 
 $$\mathfrak{g}=\bigoplus\limits_{i=1}^sI_i \text{ and } \mathfrak{h}=\bigoplus\limits_{i=1}^{s'}J_i$$
 the decompositions into indecomposable Lie subalgebras given by Proposition \ref{Decomposition en indecomposables}. Then 
 $$ \mathfrak{k}(\RR)= \left(\bigoplus\limits_{i=1}^sI_i\right) \oplus \left(\bigoplus\limits_{i=1}^{s'}J_i\right)$$
 is a decomposition into indecomposable Lie subalgebras. But the $\mathfrak{k}_i(\RR)$'s are indecomposable as well according to Proposition \ref{Algèbres absolument indécomposables}. By Proposition \ref{Decomposition en indecomposables}, there is $I \subset \{1,\ldots, r \}$ such that 
 $$\bigoplus\limits_{i \in I} \mathfrak{k}_i\left(\RR\right)\simeq \bigoplus\limits_{i=1}^sI_i = \mathfrak{g}.$$
 Since the $\mathfrak{k}_i$'s are $\RR\cap\overline{\QQ}$-Lie algebras, any basis of $\bigoplus\limits_{i \in I} \mathfrak{k}_i$ will give rise to a basis of $\mathfrak{g}$ with structure constants in $\RR \cap \overline{\QQ}$ (see Remark \ref{Remarque algebres de Lie}).
 
\end{proof}

\begin{proof}[Proof of Theorem \ref{Existence d'un réseau approximatif}, sufficiency.]
 Let $G$ be such a simply connected nilpotent Lie group. Then $G$ can be seen as the group of real points of an algebraic group $\mathbb{G}$ defined over a number field $K$. Let $Res_{K/\QQ}\mathbb{G}$ denote the Weil restriction of $\mathbb{G}$, then the group of real points $Res_{K/\QQ}\mathbb{G}(\RR)$ is a simply connected nilpotent Lie group, defined over $\QQ$ and isomorphic to a product $G \times H$ (see \cite[Sections $11$ and $12$]{springer2010linear}). As it is defined over $\QQ$, it contains a uniform lattice $\Gamma < G \times H$ according to Malcev's theorem.  
 
 Now let $L$ be the Zariski closure of $\Gamma \cap H$. As $\Gamma \cap H$ is normal in $\Gamma$, $L$ is normal in $H$. The image of $\Gamma$ in $G \times H/L$ is still a uniform lattice so we can assume that $\Gamma \cap H = \{e\}$. Then $(G,\overline{\Gamma},\Gamma)$ is a cut-and-project scheme, so $G$ contains an approximate lattice. 
\end{proof}

\begin{Rem}
 The sufficiency part of Theorem \ref{Existence d'un réseau approximatif} was already mentioned in \cite{bjorklund2016approximate}. 
\end{Rem}



\section*{Acknowledgments} 
This work is part of a Master's Thesis defended in July, 2018. I am deeply grateful to my supervisor, Emmanuel Breuillard, for his guidance and encouragements. 
I would also like to thank Tobias Hartnick for his thorough reading and valuable advice. Finally, I would like to express my very great appreciation to the anonymous reviewer for their many valuable and constructive suggestions.

\bibliographystyle{amsplain}


\begin{dajauthors}
\begin{authorinfo}[sm]
  Simon Machado\\
  Department of Pure Mathematics and Mathematical Statistics\\
  University of Cambridge\\
  Cambridge, UK\\
  sm2344\imageat{}cam\imagedot{}ac\imagedot{}uk \\
\end{authorinfo}
\end{dajauthors}

\end{document}